\documentclass{article}
\usepackage{amsmath, amssymb, amsfonts,graphicx, hyperref,enumitem,mathtools}
\newcommand{\Adeles}{\mathbf{A}}
\newcommand{\adeles}{\mathbf{A}}

\newcommand{\Q}{\mathbb{Q}}
\newcommand{\Z}{\mathbb{Z}}

\newcommand{\G}{\mathbb{G}}
\newcommand{\bbP}{\mathbb{P}}
\newcommand{\calO}{\mathcal{O}}

\newcommand{\cp}{{\cal P}}
\newcommand{\calP}{{\cal P}}
\newcommand{\fp}{\mathfrak p}

\newcommand{\br}{\operatorname{Br}}
\newcommand{\Br}{\br}
\newcommand{\ev}{\operatorname{ev}}
\newcommand{\inv}{\operatorname{inv}}
\newcommand{\gal}{\operatorname{Gal}}

\newcommand{\pic}{\operatorname{Pic}}

\newcommand{\norm}{\operatorname{Norm}}
\newcommand{\frob}{\operatorname{Frob}}
\renewcommand{\div}{\operatorname{div}}
\newcommand{\Div}{\operatorname{Div}}
\newcommand{\val}{\operatorname{v}}

\newcommand{\Xbar}{\overline{X}}

\newcommand{\thetabar}{\overline{\theta}}

\usepackage{fixltx2e}

\newcommand{\calA}{\mathcal{A}}
\newcommand{\et}{\acute{e}t}

\newcommand{\LHprod}{\prod_{i=0}^2 (x+\phi_iz+\psi_i w)}
\newcommand{\LHS}{\LHprod}
\newcommand{\RHS}{dy(x+\theta y)(x+\thetabar y)}

\newcommand{\CT}{Colliot-Th\'el\`ene}

\usepackage{tikz-cd}

\usepackage{amsthm}
\newtheorem{thm}{Theorem}[section]
\newtheorem{cor}[thm]{Corollary}
\newtheorem{lem}[thm]{Lemma}

\theoremstyle{definition}
\newtheorem*{defn}{Definition}
\newtheorem*{ex}{Example}
\newtheorem*{rmk}{Remark}

\usepackage[margin=1in]{geometry}
\usepackage{todonotes}
\begin{document}
\title{On Birch and Swinnerton-Dyer's cubic surfaces}
\author{Mckenzie West}
\maketitle
\abstract{
	\noindent In a 1975 paper of Birch and Swinnerton-Dyer, a number of explicit norm form cubic surfaces are shown to fail the Hasse Principle.  They make a correspondence between this failure and the Brauer--Manin obstruction, recently discovered by Manin.
	We generalize their work, making use of modern computer algebra software to show that a larger set of cubic surfaces have a Brauer--Manin obstruction to the Hasse principle, thus verifying the \CT--Sansuc conjecture for infinitely many cubic surfaces.
}

\section{Introduction}
Suppose $X$ is a smooth projective variety over a number field $k$.  Denote by $X(k)$ and $X(\adeles_k)$ the rational and ad\'elic points of $X$, respectively.  The natural inclusion of $k$ into $\adeles_k$ gives $X(k)\neq\emptyset\Rightarrow X(\adeles_k)\neq\emptyset$.  The variety X satisfies the \emph{Hasse principle} if the converse is true, $X(\adeles_k)\neq\emptyset\Rightarrow X(k)\neq\emptyset$.  There are many examples of varieties which do not satisfy the Hasse Principle; Cassels and Guy, in \cite{CG}, provided one of the original counterexamples,
	\begin{equation}\label{CGcubic}
		5x^3+ 12y^3+ 9z^3+ 10w^3=0.
	\end{equation}

As a cohomological generalization of quadratic reciprocity, Manin constructs, in \cite{manin1}, the \emph{Brauer Set} $X(\adeles_k)^{\br{} }$ which lies between $X(k)$ and $X(\adeles_k)$.  We say that $X$ has a \emph{Brauer--Manin obstruction to the Hasse Principle} if $X(\adeles_k)^{\br{}}=\emptyset$ while $X(\adeles_k)\neq\emptyset$.
Around this time, Birch and Swinnerton-Dyer, in \cite{BSD}, considered counterexamples to the Hasse Principle for rational surfaces via very direct arguments. They comment that Manin's method should apply and provide a brief sketch to this effect.  We will examine the cubic surfaces constructed by Birch and Swinnerton-Dyer:
	\begin{quote}
  		Let $K_0/k$ be a non-abelian cubic extension, and $L/k$ its Galois closure.  Suppose $K/k$ is the unique quadratic extension which lies in $L$.  We will assume that $(1,\phi,\psi)$ are any linearly independent generators for $K_0/k$, and $K/k$ is generated by $\theta$.  Then consider the diophantine equation given by
 			 \begin{equation}\label{BSDeqn}m\norm_{L/K}(ax+by+\phi z+\psi w) = (cx+dy)\norm_{K/k}(x+\theta y),\end{equation}
 		where the $m,a,b,c,d$ are suitably chosen $k$-rational integers.  
	\end{quote}

Birch and Swinnerton-Dyer show that as long as $a,b,c,d$ have ``certain'' divisibility properties, the projective cubic surface in $\bbP_k^3$ defined by the equation \eqref{BSDeqn} do not satisfy the Hasse Principle.  Their examples of failure are very specific, choosing values for almost all of the constants. This is done by considering a rational solution $[x\colon y\colon z\colon w]$ and examining the possible factorizations of the ideal $(x+\theta y)$ in $\calO_K$.  They find two possible reasons the Hasse Principle may fail and give an example computation of the Brauer--Manin obstruction for each.  This paper re-examines the BSD surfaces, generalizing their work using the fact that we can compute the lines of a surface using computer algebra software.

\begin{thm}\label{thm:main}
	Suppose $k$ is a number field and $X$ is as in \eqref{BSDeqn}.  
	\begin{enumerate}[ref={\thethm.\arabic*}]
		\item\label{thm:main-1} 	Let $L'/k$ be the minimal extension over which the 27 lines, or exceptional curves, on $\Xbar$ are defined.  If $9\mid [L'\colon k]$ then $\br X/\br k\simeq \Z/3\Z$ otherwise $\br X/\br k$ is trivial.
		\item\label{thm:main-2}  In the case that $\Br X/\Br k$ is trivial, we have that $X$ satisfies the Hasse Principle.
	\end{enumerate}
\end{thm}

Furthermore, this classification provides a quick method for computing the existence of Brauer--Manin obstructions for these surfaces, as in the following result.
\begin{thm}\label{thm:HasObstruction}
  Take $k$ a number field, $L/K$ unramified, and the $\phi_i$ and $\psi_i$ to be integral units in $K_0$ with the minimal polynomial of $\psi_i/\phi_i$ being separable modulo $\fp\mid 3\calO_k$.  Suppose $p\calO_k$ is a prime for which $p\calO_L=\calP_1\calP_2$ such that $p\|\theta\thetabar$ and the $\calP_i$ are primes in $\calO_L$.  Then the variety defined by 
  \[\LHS = py(x+\theta y)(x+\thetabar y)\]
  has a Brauer--Manin obstruction to the Hasse Principle.
\end{thm}

Note that if $K$ has class number 1 then we may assume $\theta$ and $\thetabar$ are coprime via a scaling of $y$ by their greatest common divisor.  In that case, we can assume $\calP_1\mid\theta$ and $\calP_2\mid\thetabar$, making this conclusion true whenever $\val_{\calP_1}(\theta)\not\equiv 0\pmod 3$.

\subsection{The Hasse Principle for Cubic Surfaces}

After the paper of Birch and Swinnerton-Dyer, \CT\ and Sansuc conjectured that the Brauer--Manin obstruction is the only obstruction to the Hasse principle for arbitrary smooth projective geometrically rational surfaces \cite[Questions j$_1$, k$_1$, page 233]{cts}.  Some motivation for this conjecture came from the study of conic bundles.

In 1987, \CT, Kanevsky and Sansuc \cite{ctks} systematically studied diagonal cubic surfaces over $\Q$ having integral coefficients up to 100, verifying the conjecture for each one of these surfaces.  They were the first to prove that the Cassels and Guy cubic \eqref{CGcubic} had a Brauer--Manin obstruction.  

This conjecture has been extended by \CT\ to all rationally connected varieties \cite{ctconj}; evidence of this generalization has been seen recently in works of Harpaz and Wittenberg (see e.g.,~\cite{harpaz-witt}).

Work by Elsenhans and Jahnel has shown general construction for lines on cubic surfaces as well as the resulting elements in $\Br k(X)$, \cite{EJ12}.  More recently, they prove that cubic surfaces violating the Hasse principle are Zariski dense in the moduli space of all cubic surfaces, \cite{EJ15-2}.  This paper, on the other hand, wishes to show explicit computations for the particular family described by Birch and Swinnerton-Dyer.

\subsection{Outline}
The notation for the paper will be fixed in section \ref{notation}.  We will explicitly describe the Brauer group for the BSD cubic surfaces in section \ref{brauerGroup} and subsequently prove Theorem \ref{thm:main}.  This computation will exploit the exceptional geometry of cubic surfaces and the previous results of Corn, \cite{corn}, and Swinnerton-Dyer, \cite{sd93,sd99}.

In section \ref{inv}, we prove the existence of an ad\'elic point for a family of surfaces followed by general computations of the Brauer set.  Theorems \ref{zeros} and \ref{obs1} show that we only need to consider certain primes which divide the coefficient $d$.  This section concludes with a proof of Theorem \ref{thm:HasObstruction}.

Lastly, in section \ref{examples}, we first look back at an example given in \cite{BSD} and verify that its obstruction is given by the results of section \ref{inv}.  A second example with a Brauer--Manin obstruction given by two non-zero invariant summands is then presented.

\section{Setup and background}\label{notation}
Let $k$ be a number field with absolute Galois group $G_k$. Take $L/k$ any Galois extension with $\gal(L/k)\simeq S_3$. Fix $K/k$ as the unique quadratic extension of $k$ in $L$.  Let $\calO_F$ be the ring of integers for the field $F$.
\begin{lem}
  Every BSD cubic \eqref{BSDeqn} equation is isomorphic to one of the form
  \begin{equation}\label{X}\LHS = \RHS,\end{equation}
  where $d\in \calO_k$, and $\{\phi_0,\ \phi_1,\ \phi_2\}$, $\{\psi_0,\ \psi_1,\ \psi_2\}\subseteq\calO_L$ and $\{\theta,\ \thetabar\} \subseteq \calO_K$ are respective Galois conjugates over $k$ with $(1,\phi_i,\psi_i)$ being a $k$-basis for a degree 3 extension of $k$.
\end{lem}
\begin{proof}
  There is an isomorphism of varieties given by \[[x\colon y\colon z\colon w]\mapsto \left[ax+by\colon \frac{cx+dy}{(d-c\theta)(d-c\thetabar)}\colon z\colon w\right],\]
  from the surface in $\bbP^3_k$ defined by \eqref{BSDeqn}, to the surface in $\bbP^3_k$ defined by
  \begin{equation*}
    m(ad-bc)^2\prod_{i=0}^2(x+\phi_i z+\psi_i w)=((d-c\theta)(d-c\thetabar))^2y(x+\theta' y)(x+\thetabar' y),
  \end{equation*}
  where $\theta'=(-b+a\theta)(d-c\thetabar)$.
  A subsequent isomorphism given by scaling variables results in \eqref{X}.
\end{proof}

Let $X$ to be the closed subvariety of $\bbP_k^3$ defined by \eqref{X} and let $\Xbar$ be the base change of $X$ to the separable closure of $k$. Take $\pic \Xbar$ to be the Picard group of $\Xbar$.  For a fixed $\calA \in\br X\coloneqq \operatorname{H^2_{\et }}(X,\G_m)$, there is a commutative diagram
\begin{center}
  \begin{tikzcd}
    \
    &X(k)\arrow[hook]{r}\arrow{d}{\ev_{\calA}}
    &X(\Adeles_k)\arrow{d}{\ev_{\calA}}\arrow{rd}{\Phi_\calA}\\
    0\arrow{r}
    &\br k\arrow{r}
    &\oplus_v \br k_v\arrow{r}[swap]{\inv}
    &\Q/\Z\arrow{r}
    &0
  \end{tikzcd}
\end{center} 
where $\ev_\calA$ is the specialization of $\calA$ and $\inv=\sum_v\inv_v$ is the \emph{invariant map}.  The \emph{Brauer Set} is defined as $X(\adeles_k)^{\br }\coloneqq\bigcap_{\calA\in \br X}\Phi_\calA^{-1}(0)$.  

There is an inclusion of $\br X$ into $\br k(X)$, so elements of $\br X$ can be realized as Azumaya Algebras over the field $k(X)$.  Moreover, Azumaya Algebras over a field such as $k(X)$ are exactly the central simple algebras over $k(X)$.  Assume $F_1/F_2$ are fields, $\gal(F_1/F_2)=\langle\sigma\rangle$ is cyclic of order $n$ and $a\in F_2^*$, then the \emph{cyclic algebra} $(F_1/F_2,a)$ is defined to be the quotient $F_1[T]_\sigma/(T^n-a)$.  Here $F_1[T]_\sigma$ is the twisted polynomial ring, i.e., $Tb=\sigma(b)T$ for all $b\in F_1$.  Further such cyclic algebras are central simple algebras, so $(F_1/F_2,a)$ can be used to represent an equivalence class in $\Br F_2$.

\section{Proof of Main Theorem}\label{brauerGroup}
Since $X$ is rational, $\ker (\br X\to\br\Xbar)=\br X$, \cite[Thm. 42.8]{manin2}, and there is an isomorphism given by the Serre Spectral Sequence and Tate correspondence, 
\begin{equation}\label{quotientIso}\br X/\br k\xrightarrow{\sim} \operatorname{H^1}(G_k,\pic\Xbar).\end{equation} Moreover, $\Phi_\calA$ factors through this quotient.  Therefore it is sufficient to calculate this finite group rather than determining the entirety of $\br X$.  

Swinnerton-Dyer and Corn provide a classification of possible $\operatorname{H^1}(G_k,\pic\Xbar)$ via the following structures existing within the set of exceptional curves on $\Xbar$.
\begin{defn}
	A \emph{double-six} on $X$ is a set $\{L_1,\dots,L_6\}\cup \{M_1,\dots,M_6\}$ of 12 exceptional curves on $X$ such that the $L_i$ are pairwise skew, the $M_i$ are pairwise skew, and $L_i$ intersects $M_j$ exactly when $i\neq j$.
	
	  A \emph{nine} on $X$ is a set consisting the three skew curves together with six curves each of which intersect exactly two of those three. A \emph{triple-nine} is a partitioning of the 27 exceptional curves on $X$ into three nines.
\end{defn}
\begin{lem}\label{lem:classify}
	Let $X$ be a cubic surface defined over the number field $k$.
	\begin{enumerate}
		\item (\cite[Lem.~1.3.9]{corn}, \cite[Lem.~1]{sd93}) Elements of $\br X/\br k$ of order 2  correspond to $G_k$-stable double-sixes such that neither six is itself $G_k$-stable, no pair $\{L_i,M_i\}$ is $G_k$-stable and no set of three such pairs is $G_k$-stable.
		\item (\cite[Lem.~1.3.22]{corn}, \cite[Lem.~6]{sd93})\label{lem:classify-order3} There is a one-to-one correspondence between order 3 elements of $\br X/\br k$ and triple-nines on $X$ such that each nine is itself $G_k$-stable but no set of three skew lines within any nine is $G_k$-stable.
		\item (\cite[Lem.~5]{sd93}) In order for $\br X/\br k$ to have more than 2 elements of order 3 it is necessary and sufficient that the field of definition for the exceptional curves be of order 3 over $k$.
		\item (\cite[Thm.~1.4.1]{corn},\cite{sd93}) The quotient $\br X/\br k$ is isomorphic to one of  $\{1\}, \Z/2\Z, (\Z/2\Z)^2,\Z/3\Z,$ or $(\Z/3\Z)^2$.
	\end{enumerate}
\end{lem}

\begin{proof}[Proof of \ref{thm:main-1}.]
  There are 9 lines, $L_{i,j}$, defined by $0=x+\phi_i z+\psi_i w$ and 
  \[0=\left\{\begin{array}{cc} 
  y&\textup{if }j=0,\\
  x+\theta y &\textup{if } j=1,\\
  x+\thetabar y & \textup{if } j=2,
  \end{array}\right.\] 
  and 18 lines, $L_{(i,j,k), n}$ given by $z=Ax+By$ and $w=Cx+Dy$ such that $A,\ B,\ C,$ and $D$ satisfy the system of equations,
  \begin{equation}\label{bigLines}\left\{\begin{array}{l}
    1+A\phi_i +C\psi_i\ =\ 0,\\
    \theta(1+A\phi_j+C\psi_j)=(B\phi_j+D\psi_j),\\
    \thetabar(1+A\phi_k+C\psi_k)=(B\phi_k+D\psi_k),\\
    (B\phi_0+D\psi_0)(B\phi_1+D\psi_1)(B\phi_2+D\psi_2)=d\theta\thetabar.
    \end{array}\right.\end{equation}
	Here the value of $n$ in $L_{(i,j,k), n}$ corresponds to the three possible solutions of the system \eqref{bigLines} for a fixed triple $(i,j,k)$.
	
  As a result of the definition and intersection properties of lines on cubic surfaces, we can write a nine as a $3\times 3$ matrix
  \begin{equation*} 
  \begin{pmatrix}
  	\ell_{1,1}&\ell_{1,2}&\ell_{1,3}\\
  	\ell_{2,1}&\ell_{2,2}&\ell_{2,3}\\
  	\ell_{3,1}&\ell_{3,2}&\ell_{3,3}
  \end{pmatrix},
  \end{equation*}
  such that the intersection pairing has
	  \begin{equation*}
		  (\ell_{i,j},\ell_{m,n})=\begin{cases} 1& \textup{if }i\neq m \textup{ and }j\neq n,\\
		  -1& \textup{if }i=m \textup{ and } j=n,\\
		  0& \textup{otherwise}.
		 \end{cases}
		\end{equation*}
  
  Let $L'/k$ be the field of definition for the 27 lines.  
  A triple-nine for which the individual nines are fixed by $G_k$ is
  \begin{equation}\label{triplenine}
  {
  \begin{pmatrix}
    L_{0,0}&L_{1,1}&L_{2,2}\\
    L_{1,2}&L_{2,0}&L_{0,1}\\
    L_{2,1}&L_{0,2}&L_{1,0}
  \end{pmatrix},
  \begin{pmatrix}
    L_{(0,1,2),0}&L_{(0,1,2),1}&L_{(0,1,2),2}\\
    L_{(1,2,0),0}&L_{(1,2,0),1}&L_{(1,2,0),2}\\
    L_{(2,0,1),0}&L_{(2,0,1),1}&L_{(2,0,1),2}
  \end{pmatrix},
  \begin{pmatrix}
    L_{(0,2,1),0}&L_{(0,2,1),1}&L_{(0,2,1),2}\\
    L_{(1,0,2),0}&L_{(1,0,2),1}&L_{(1,0,2),2}\\
    L_{(2,1,0),0}&L_{(2,1,0),1}&L_{(2,1,0),2}
  \end{pmatrix}.
   }
   \end{equation}
  The Galois group $G_k$ permutes the first nine, fixing no skew triple.  The rows of the second two nines will be permuted via the permutation action on the roots $(\phi_0,\phi_1,\phi_2)$.  The action of $G_k$ on the columns of the second nines will determine whether or not any skew triple is fixed.  If $3\mid[L'\colon L]$, then the columns of the second two nines are permuted so by Lemma \ref{lem:classify} $\operatorname{H^1}(G_k,\pic\Xbar)\simeq \Z/3\Z$.  Otherwise $9\nmid[L'\colon k]$ and the columns in the second two nines are fixed.  As the first nine is the only $G_k$-stable nine containing these exceptional curves, there are no other possible triple nines and thus there are no elements of order 3 in $\br X/\br k$.  Furthermore, any $G_k$-stable double six will have a set of 3 skew pairs which are $G_k$-stable.  Therefore, Lemma \ref{lem:classify} the only possibility is for $\br X/\br k\simeq\{1\}$.
\end{proof}

\begin{proof}[Proof of \ref{thm:main-2}]
	If $\br X/\br k$ is trivial, then the triple nine as in \eqref{triplenine} will have enough fixed skew triples to build a set of six skew lines which is $G_k$-stable.  In particular take two disjoint columns from the second two nines.   We blow down $X$ along these six skew lines to obtain a degree 9 del Pezzo surface $X'$ defined over $k$.  It is well-known that degree 9 del Pezzo surfaces satisfy the Hasse principle.  So by the Lang-Nishimura lemma $X$ must also satisfy the Hasse principle (cf.~\cite{Nis55}).
\end{proof}

The map in \eqref{quotientIso} is generally difficult to invert.  We achieve this via the following result of Swinnerton-Dyer and Corn.

\begin{lem}[{\cite[Lem. 2]{sd99}, \cite[Prop. 2.2.5]{corn}}]\label{orderThree}
  Non-trivial elements of $\operatorname{H^1}(G_k,\pic \Xbar)[3]$ correspond to $\calA\in\br X$ such that $\calA\otimes_k K = (L(X)/K(X),f)$ in the group $\br X_K/\br K$, where $\div(f)=N_{L/K}(D)$ in $\Div X_L$ for some non-principal divisor $D$.
\end{lem}
\begin{cor}
  If\: $\operatorname{H^1}(G_k,\pic\Xbar)\simeq \Z/3\Z$ then it is generated by an algebra $\calA$ such that $\calA\otimes_k K\simeq \left(L(X)/K(X),\frac{x+\theta y}{y}\right)$.
\end{cor}
\begin{proof}
  Take $D=L_{0,0}+L_{1,1}+L_{1,0}-\div(y)$.  Clearly $D$ is not principal, as the intersection pairing  between $D$ and a line, say $L_{1,1}$, which can be chosen as a generator of $\pic\Xbar$ over $\Z$ is non-zero, that is $(D,L_{1,1})=0-1+1-1\neq0$.  Then
  \begin{align*}
    \norm_{L/K}(D) &= (L_{0,0}+L_{1,1}+L_{1,0}) + (L_{1,0}+L_{2,1}+L_{2,0})+(L_{2,0}+L_{0,1}+L_{0,0}) - 3\div(y)\\
    &= L_{1,1}+L_{2,1}+L_{0,1}-\div(y)\\
    &= \div(x+\theta y)-\div (y)\\
    &= \div\left(\frac{x+\theta y}{y}\right).\tag*{\qedhere}
  \end{align*}
\end{proof}

\section{Invariant map computations}\label{inv}
Since the $\calA\in\br X/\br k$ are explicit, one may compute the map $\phi_\calA$ more easily.  However, before doing so, we would like to verify the existence of an ad\'elic point.
\begin{lem}\label{adeles}
  In addition to the setup of section \ref{notation}, assume the following are true:
  \begin{enumerate}
  \item $L/K$ is unramified, 
  \item $\phi_0\phi_1\phi_2=\psi_0\psi_1\psi_2=\pm1$,
  \item\label{adeles3} the minimal polynomial for $\psi_i/\phi_i$ over $k$ is separable modulo $\fp\mid 3\calO_k$, and
  \item if $\fp\mid d\calO_L$ with $\fp\calO_L=\calP_1\calP_2$ then $\val_1(d)\leq \val_1(\theta)$ with $\val_1(\thetabar) = 0$, equivalently $\val_2(d)\leq \val_2(\thetabar)$ with $\val_2(\theta)=0$, where $v_i$ is the valuation corresponding to $\calP_i$. 
  \end{enumerate}
  Then $X(\Adeles_k)\neq\emptyset$.
\end{lem} 
Notice that the fourth requirement includes the stipulation that $\val_1(\thetabar)=0$ or $\val_2(\theta)=0$, this will ensure that $p\nmid \theta\thetabar$.
\begin{proof}
  In most cases, the scheme given by $X\cap V(x)$ will be a genus 1 curve and will subsequently have a $k_\fp$ point by the Hasse bound.  This will be the case whenever $\fp\nmid 3d\calO_k$.

  Suppose $\fp\mid 3\calO_k$ and $\fp\nmid d\calO_k$.  Then $X\cap V(x) \to \bbP^1$ defined by $[0\colon y\colon z\colon w]\mapsto[z\colon w]$ is one-to-one and surjective on $k_\fp$ points.	Assumption \ref{adeles3} provides that at least one of these points is smooth.

  For the primes $\fp$ of $k$ dividing $d$, to show $X(k_\fp)\neq\emptyset$, it will suffice to find a $K_{\calP}$ point for each prime $\calP$ of $K$ dividing $\fp$.  This is a result of the fact that on cubic surfaces the existence of $k_\fp$-rational points is equivalent to that for quadratic extensions of $k_\fp$ (cf.~\cite[Lem. 1.3.25]{corn}).

  If $\calP\mid d\calO_K$ and $\calP$ splits over $L$ then $X_\calP$ is the union of 3 lines all defined over $K_\calP/\calP$ and has many $K_\calP$ points.

  Lastly, suppose $\calP\mid d\calO_K$ and $\calP$ remains prime in $L$.  Then we are in the case of $\val_\calP(d)\leq \val_\calP(\theta)$.  
  Then consider the cubic surface $X'$ in $\bbP^3_k$ defined by the equation,
  \[d\LHS - y\left(x+(\theta/d)y\right)(dx+\thetabar y)=0,\]
  which is isomorphic to $X$.  Note that this equation has $\calO_{\calP_1}$ coefficients since $\val_1(\theta)\geq \val_1(d)$.  Modulo $\calP_1$, the defining equation for $X'$ becomes
  \[X'_1\colon  \thetabar_1y^2\left(x+(\theta/d)_1y\right),\]
  where $\thetabar_1$ and $\left(\theta/d\right)_1$ are the restriction of the respective constants to the quotient $\calO_{\calP_1}/\calP_1$.  The surface $X'_1$ has a smooth point $[\theta_1/d\colon  -1\colon 1\colon 1]$ which will lift to a $K_{\calP_1}$ point, $[x_0\colon y_0\colon z_0\colon w_0]\in X'(K_{\calP_1})$.  Via the isomorphism, we have $[dx_0\colon y_0\colon dz_0\colon dw_0]\in X(K_{\calP_1})$.
\end{proof}

\begin{rmk}
  In the case of $\val_1(d)>\val_1(\theta)$, a similar argument can be made with the additional assumption of the surjectivity of the cube map in $\calO_{\calP_1}/\calP_1$.
\end{rmk}

Of course Lemma \ref{adeles} is not comprehensive; there are surfaces in the class which have ad\'elic points but do not satisfy the conditions listed above.  The intention of this lemma is to provide proof that there are indeed infinitely many surfaces of this form which have an ad\'elic point.

There is a classical formula for $\inv_v$ provided $L_v/K_v$ is unramified given by the local Artin map.
That is, for all places $v$ unramified in $L/K$,
\begin{equation*}
  \inv_v\left(\left(L_v/K_v,f(P_v)\right)_\sigma\right) = \frac{ij}{k}\mod 1,
\end{equation*}
where $i=\val_v(f(P))$, $\sigma^j=\frob_{L_v/K_v}$, and $k=[L_v\colon K_v]$ (cf.~\cite[XIV.2]{ser}).
\begin{thm}\label{zeros}
  Assume the notation of section \ref{notation}.  Suppose $v_0$ is a finite place of $K$ which is unramified in $L/K$ such that $\val_{v_0}(d)= 0\pmod 3$, and that $\theta$ or $\thetabar$ has valuation 0.  Then $\inv_{v_0}(\calA_K(P_{v_0}))=0$ for all $(P_v)\in \prod_v X(K_v)$.  Moreover, $\inv_\infty (\calA_K(P_\infty))=0$.
\end{thm}
\begin{proof}
  (The structure of this proof follows that of \cite[III.5.18]{jahnel}.)
  In the infinite case, we must have $\inv_\infty(\calA_K(P_\infty))=0$, as $[L\colon K]=3$ and $\inv_\infty(\calA_K(P_\infty)) = 0$ or $1/2$.

  Suppose that $v$ splits completely in $L$.  Then $L_v=K_v$ and $(L_v/K_v,f(P_v))$ is trivial.  Thus we must have $\inv_v(\calA_K(P_v))=0$.

  If $v$ remains prime in $L$ then $[L_v\colon K_v]=3$.  Take $P_v= [x_0\colon y_0\colon z_0\colon w_0]\in X(K_v)$.  Via scaling, assume that $x_0,\ y_0,\ z_0$ and $w_0$ are integral and at least one has valuation 0. Since $\prod_{i=0}^2(x_0+\phi_i z_0+\psi_i w_0)$ is a norm from $L$ to $K$, $y_0=0$ would imply $x_0=z_0=w_0=0$, which is not possible.  Thus $y_0\ne 0$.  In particular, $f(P_v)=\frac{x_0+\theta y_0}{y_0}$ is defined for all $P_v\in X(K_v)$.  

  For simplicity, set $\val=\val_v$ and $N=\prod_{i=0}^2(x_0+\phi_i z_0+\psi_i w_0)$. If $\val(N)=0$, then $\val(y_0)=\val(x_0+\theta y_0)=0$.  Hence $\inv(\calA_K(P_v))=0$.
  On the other hand, suppose $\val(N)>0$.  Here the restriction of $N$ modulo $v$ is a norm on the residue class field of $L_v$ to that of $K_v$.  Thus $N$ having positive valuation implies that the restriction $\overline{N}=0$.  Hence $x_0,z_0,w_0\cong 0\pmod v$ so $\val(y_0)=0$. In fact, $3\mid \val(N)$.  Thus, $\val(d)+\val(x_0+\theta y_0)+\val(x_0+\thetabar y_0)\equiv 0\mod 3$.  However, $\val(x_0+\theta y_0)=0$ or $\val(x_0+\thetabar y_0)=0$, since $v$ does not divide both $\theta$ and $\thetabar$.  In particular, $\val(x_0+\theta y_0)\equiv 0 \pmod 3$ and 
  \begin{equation*}
  	\displaystyle\inv_v(\calA_K(P_v))=0.\qedhere
  \end{equation*}
\end{proof}

\begin{rmk}
  This result should be expected, because unramified primes of good reduction produce a trivial invariant computation.
\end{rmk}
Given the result of Theorem \ref{zeros}, in all cases where $L/K$ is unramified, we simply need to consider the places of $k$ over which $d$ has valuation that is non-zero modulo 3.  The following theorem provides a sample of the types of Brauer--Manin obstructions we may now construct for the surfaces $X$.
\begin{thm}\label{obs1}
  With the notation of section \ref{notation}.  Suppose $L/K$ is unramified.  Fix $\theta$ so that no primes of $\calO_K$ divide both $\theta$ and $\thetabar$.  Let $\fp$ be a prime of $\calO_k$ such that $\val_{\fp}(d)=n$ for some $n\not\equiv0\pmod 3$  and $\fp\mid \theta\thetabar$.  Suppose all other primes dividing $(d)$ split in $L/K$.  If $X(\Adeles_K)\neq\emptyset$ and $\fp=\cp_1\cp_2$ in $\calO_K$, each of which is inert in $\calO_L$, then $\sum_v \inv_v(\calA_K(P))\neq 0$.
\end{thm}
\begin{proof}
  From the statement and proof of Theorem \ref{zeros}, we need only consider the primes $\cp_1$ and $\cp_2$ of $\calO_K$ that lie above $\fp$.  Via our assumption that no primes divide both $\theta$ and $\thetabar$, we can assume that $\cp_1\mid\theta$ and $\cp_2\nmid\theta$.  Take $\val_i=\val_{\cp_i}$ to be the respective valuation maps.  As $\val_i(\RHS)\ >\ 0$, we must be in the case that $\val_i(x_0),\ \val_i(w_0),\ \val_i(z_0)\ >\ 0$ and $\val_i(y_0)=0$.  Then $\val_2(x+\theta y)=0$, so $\inv_{\cp_2}(\calA(P))=0$.  On the other hand $\val_1(x+\theta y)= \val_1(\RHS)-\val_1(d) \equiv -\val_1(d)\mod 3$.  In particular $\inv_{\cp_1}(\calA(P))=1/3$ or $2/3$.  Thus \[\sum\inv_v(\calA(P))=\inv_{\cp_1}(\calA(P))\neq 0.\qedhere\]
\end{proof}
This theorem provides a jumping off point for similar results.  One may consider the case where more places divide $d$, and examples of most forms can be computed immediately.
\begin{proof}[Proof of \ref{thm:HasObstruction}.]
	Lemma \ref{adeles} guarantees that $X(\Adeles_k)\neq\emptyset$.  Then Theorem \ref{zeros} implies that the only possible non-zero summand in the invariant map is that corresponding to $p$ while Theorem \ref{obs1} implies that this single invariant map is non-zero. Thus $\Phi_{\calA}(P)\neq0$ for every $P\in X(\Adeles_k)$ and $X$ has a Brauer--Manin obstruction to the Hasse Principle.
\end{proof}

\section{Examples}\label{examples}
Examples that fit the situation of this Theorem \ref{obs1} are easy to come by.  Given any $L/K$ unramified we can find many such $\theta$.  Then it is a quick check via Hensel's Lemma and the Weil Conjectures to show that there is an ad\'elic point.  In fact the original example of BSD fits this case.
\begin{ex}
  Suppose $\theta' = \frac{1}{2}(1+\sqrt{-23})$ and $\phi_i$ so that $\phi_i^3=\phi_i+1$ and $\psi_i=\phi_i^2$. Define $X_{BSD}$ by 		\[2\prod_{i=0}^2(x+\phi_iz+\phi_i^2 w)=(x-y)(x+\theta' y)(x+\thetabar' y).\]
  Via the isomorphisms above, we have the isomorphic $X$ given by 
  \[\prod_{i=0}^2(x+\phi_i z+\psi_i^2 w)=32y(x+\theta y)(x+\thetabar y),\]
  where $\theta = -\theta '-6$.

  We find that $X$ has ad\'elic points but no rational points.
  Moreover, $X$ has a Brauer--Manin obstruction to rational points as described in Theorem \ref{obs1}.
\end{ex}

There are few published examples where the invariant map has two or more non-zero summands.  Given the theorems above, examples of this can be found quickly.
\begin{ex}
  Suppose the $\phi_i$ satisfy $\phi_i^3+\phi_i+1=0$ and $\theta,\ \thetabar$ are the roots of $T^2-4T+35$.

  Then 
  \[X\colon \prod_{i=0}^2 (x+\phi_iz+\psi_iw)\ =\ 5^2\cdot 7y(x+\theta y)(x+\thetabar y),\]
  has a Brauer--Manin obstruction to the Hasse Principle with the invariant map being \[1/3\ +\ 1/3\hskip .2cm \textup{ or }\hskip .2cm 2/3\ +\ 2/3,\] depending on the choice of algebra $\calA$.
\end{ex}

\bibliography{cspaper}
\bibliographystyle{plain}
\end{document}